\documentclass[twoside,a4paper,12pt]{article}
\usepackage{amssymb, amsmath, amsthm}
\usepackage{fancyhdr}
\usepackage{color}

\newdimen\myMargin
\textwidth \paperwidth
\advance\textwidth -2\myMargin
\textheight \paperheight
\advance\textheight -2\myMargin
\advance\oddsidemargin \myMargin
\advance\evensidemargin \myMargin
\advance\topmargin \myMargin
\advance\textheight -17 true mm
\advance\topmargin 12.5 true mm

\newcommand{\Aalg}{{\cal A}}
\newcommand{\A}{{\cal A}}
\def\B{{\cal B}}

\newcommand{\K}{\mathbb K}
\newcommand{\LL}{\mathbb L}
\newcommand{\R}{\mathbb R}
\newcommand{\Q}{\mathbb Q}
\newcommand{\C}{\mathbb C}

\newcommand{\Z}{\mathbb Z}

\newcommand{\ve}[1]{\mathbf{#1}}

\newcommand{\rank}{\mathrm{rank}}

\newtheorem{theorem}{Theorem}
\newtheorem{lemma}[theorem]{Lemma}
\newtheorem{corollary}[theorem]{Corollary}

\theoremstyle{remark}
\theoremstyle{remark}\newtheorem{remark}[theorem]{Remark}

\title{Splitting full matrix algebras over algebraic number fields}

\author{\normalsize
  \begin{minipage}{0.3\linewidth}
    \large
    G\'abor Ivanyos \\
    \footnotesize
    Computer and Automation Research 
Institute, Hungarian Acad. Sci; \\
    \texttt{Gabor.Ivanyos@sztaki.hu} \\
    \normalsize
  \end{minipage}
  \qquad
  \begin{minipage}{0.3\linewidth}
    \large
    Lajos R\'onyai \\
    \footnotesize
    Computer and Automation Research 
     Institute, Hungarian Acad. Sci. \\
    Dept. of Algebra, Budapest 
    Univ. of Technology and Economics \\
    \texttt{lajos@ilab.sztaki.hu}
    \normalsize
  \end{minipage}
 \qquad
  \begin{minipage}{0.3\linewidth}
    \large
    Josef Schicho \\
    \footnotesize   
    Johann Radon Institute of Computational and Applied Mathematics;\\
Austrian Academy of Sciences\\
 \texttt{josef.schicho@oeaw.ac.at}
    \normalsize
  \end{minipage}
  \vspace{0.5cm}
}

\begin{document}

\thispagestyle{empty}

\maketitle

\footnotetext{
The work of Ivanyos and R\'onyai was supported in part by OTKA grants 
NK 72845, K77476, and K77778. 
}

%==========================================================================

\begin{abstract}
Let $\K$ be an algebraic number field of degree $d$ and discriminant
$\Delta$
over $\Q$. Let 
$\A$ be an associative algebra over $\K$ given by structure constants such
that $\A\cong M_n(\K)$ holds for some positive integer $n$. Suppose that
$d$, $n$ and $|\Delta|$ are bounded. Then an isomorphism
$\A \rightarrow M_n(\K)$ 
can be constructed by a polynomial time ff-algorithm. An ff-algorithm is a
deterministic procedure which is allowed to call oracles for factoring
integers and factoring univariate polynomials over finite fields. 

As a consequence, we obtain a polynomial 
time  ff-algorithm to compute ismorphisms of central simple algebras of
bounded degree over
$\K$.   
\end{abstract}

%%%%%%%%%%%%%%%%%%%%%%%%%%%%%%%%%%%%%%%%%%%%%%%%%%%%%%%%

\section{Introduction}

In this paper we consider the following algorithmic problem, which we call 
{\em explicit isomorphism problem}: 
{\em let $\K$ be an algebraic number field, $\A$ an associative algebra 
over $\K$. Suppose that
$\A$ is isomorphic to the full matrix algebra $M_n(\K)$.
Construct explicitly an isomorphism $\A\rightarrow M_n(\K)$. Or,
equivalently, give an irreducible $\A$ module.}

\medskip
Recall that for an algebra ${\A}$ over a field $\K$ 
and a $\K$-basis   $a_1, \ldots ,a_m$ of ${\A}$ over $\K$ 
the products $a_ia_j$ can be expressed as
linear combinations of the $a_i$
$$ a_ia_j=\gamma _{ij1}a_1+\gamma _{ij2}a_2+\cdots +\gamma _{ijm}a_m. $$
The elements $\gamma _{ijk}\in \K$ are called structure constants. In
this paper an algebra is considered to be given as a collection of  
structure constants. The usual representation of a number field $\K$ over 
$\Q$ with the minimal polynomial $f\in \Z[x]$ of an algebraic integer
$\alpha\in \K$ with $\K=\Q(\alpha)$ can also be considered this way.

\medskip

For basic definitions and facts
from the theory of finite dimensional associative
algebras the reader is referred to \cite{Pi} and \cite{Re}.
Let $\A$ be a finite dimensional associative algebra over $\K$, which 
is either a finite field or an algebraic number field. In
\cite{FR} and \cite{Ro2}  polynomial time algorithms were proposed
for the computation of the radical ${\rm Rad}({\A})$, and for the computation 
of the Wedderburn
decomposition (the minimal two-sided ideals) of the semisimple part
${\cal A}/{\rm Rad}({\cal A})$. The algorithm for the Wedderburn
decomposition is probabilistic (Las Vegas) in the finite case, the 
others are deterministic.
Alternative methods, improvements and related results 
have been obtained in \cite{Eb1}, \cite{Eb2}, \cite{CIW}, \cite{EG1}, \cite{EG2}, 
\cite{EG3}, \cite{GI}, \cite{Iv1}, \cite{Iv2},\cite{Ac1}. A recent survey is
\cite{Bre1}.

To obtain a decomposition of $\A$ into minimal left ideals, one has to be
able to solve 
the explicit isomorphism problem for simple algebras over $\K$. In
\cite{Ro2} this was shown to be possible in randomized polynomial time when 
$\K$ is finite. This method was derandomized recently in \cite{IKRS} in the case 
when the dimension of $\A$ over $\K$ is bounded. 
In \cite{Ro1} and \cite{Vo1} evidence (randomized reduction) 
is presented, that over algebraic 
number fields the explicit isomorphism problem problem is at least as
difficult as the task of factoring integers, a problem not known to be
amenable to polynomial time algorithms. 
For simple algebras over a number field  $\K$ polynomial time 
Las Vegas algorithms
were given in \cite{Eb1} and \cite{BR} to find a number field $\LL\supseteq
\K$ such that $\A\otimes _{\K}\LL\cong M_n(\LL)$ for a suitable $n$, together 
an explicit representation of the isomorphism. In \cite{Eb3} a real version
was established: if $\K\subset \R$, and $\A$ splits over $\R$, then it can be achieved that $\LL\subset
\R$. These results were  derandomized in part in \cite{Ro4}, and completely 
in \cite{GI}.

Following \cite{Ro3} we recall the notion of an {\em ff-algorithm}. It is an
algorithm which is allowed to call an oracle for two types of subproblems.
These are the problem of factoring integers, and the problem of factoring
polynomials over finite fields. We have no deterministic polynomial time 
algorithms for these problems (but the latter one admits polynomial time
randomized algorithms). In both cases the cost of the oracle call is the
length of the input to the call. 

In \cite{Ro3} the problem of deciding if 
$\A\cong M_n(\K)$ holds for an algebra $\A$ over a number field $\K$ was
shown to be in $NP\cap coNP$. The proof relies on properties of maximal
orders $\Lambda \leq \A$ for central simple algebras $\A$ over $\K$. Maximal
orders  are in many ways analogous to the full ring of algebraic integers in
$\K$. The principal result of \cite{IR} is a polynomial time ff-algorithm to 
construct maximal orders in simple algebras over $\Q$. A very similar
algorithm is presented in \cite{NS}. In \cite{Vo1} a more
direct method is given for quaternion algebras.

Several of the algorithms mentioned here have implementations in the
computer algebra system Magma, see
for example \cite{Ma}.

We mention also a somewhat surprising application of the algorithms for
orders: they have been applied in the construction and analysis of high 
performance space time block codes for wireless communication, see \cite{HLRV},
\cite{Se1}. In fact, in addition to an application of the algorithm of 
\cite{IR}, in \cite{HLRV} an improvement is suggested for the orders relevant 
there. 

The main result of this paper is a polynomial time ff-algorithm for the case
when $\A$ is a central simple algebra of bounded dimension 
over a small extension field $\K$ of $\Q$. 
This was known before only 
in the smallest nontrivial case $\dim _\Q \A=4$, see \cite{ISz} and the more
recent papers \cite{CR}, \cite{Si1}, \cite{Vo1}. More
precisely we have the following.

\begin{theorem} \label{main}
Let $\K$ be an algebraic number field of degree $d$ and discriminant $\Delta$ 
over $\Q$. Let 
$\A$ be an associative algebra over $\K$ given by structure constants such
that $\A\cong M_n(\K)$ holds for some positive integer $n$. Suppose that 
$d$, $n$ and $|\Delta|$ are bounded. Then an isomorphism 
$\A \rightarrow M_n(\K)$ 
can be constructed by a polynomial time ff-algorithm. 
\end{theorem}

We remark, that the algorithm of Theorem \ref{main} gives an explicit 
isomorphism even if we do not assume that $\log |\Delta|$, $n$, and  $d$
are bounded. However, the running time then may be exponential in these 
parameters. This holds also for the algorithmic applications given in the 
last section of the paper.

\medskip

In addition to computational representation theory where the problem 
naturally originates from, the explicit isomorphism problem arises 
also in connection with computational problems of arithmetic geometry: 
in a series of seminal papers \cite{CFNSS1}, \cite{CFNSS2}, and 
\cite{CFNSS3} the $n$-Selmer group of an elliptic curve $E$ over a number
field $\K$ is studied. 
A method is developed to represent the elements of the Selmer group as 
genus one normal curves of degree $n$. One of the key ingredients 
of their method is to solve the explicit isomorphism problem for 
$M_n(\K)$.  In \cite{CFNSS3} an algorithm is outlined for the explicit
isomorphism problem over $\K=\Q$, and is detailed for the cases $n=3,5$. 
Our approach is based on similar ideas.

An algorithm for explicit isomorphisms is useful also for computing 
parametrizations in algebraic geometry: 
\cite{CR} considers parametrizations of conics, and \cite{GHPS} gives
algorithms for rational parametrization of Severi-Brauer surfaces.   
In fact, in \cite{GHPS} an algorithm is given which solves the explicit
isomorphism problem when $\A\cong M_3 (\Q)$. This, however, uses a procedure
for solving norm equations whose complexity was not clear so far. For
example it was not known if they can be solved in ff-polynomial time. 
The case $\A\cong M_4(\Q)$ is treated similarly in \cite{Pi1}.

The organization of the paper is as follows. First, in Section 2  we prove 
Theorem \ref{main} in the simpler case $\K=\Q$. This combines the approach
of Fisher \cite{AL} (that is used in \cite{CFNSS3} as well), which considers 
a real embedding of $\A$, 
with an application of 
Minkowski's theorem on convex bodies, and with approximate lattice 
basis reduction. In the next section the argument is
extended to number fields. An important role is played here by the
traditional map in algebraic number theory which maps $\K$ into $\R^d$, 
see Section 13, Chapter I. in \cite{Jan}. 

In the last section two applications are presented. One of these is 
a polynomial time  ff-algorithm to compute ismorphisms of central simple 
algebras of bounded degree over $\K$.   

\medskip

\noindent
{\bf Acknowledgement.} We are grateful to G\'eza K\'os and S\'andor Z. Kiss for 
discussions on the subject. We thank Jacques-Arthur Weil for calling our
attention to \cite{AL}.

\section{Full matrix algebras over $\Q$}  

Here we consider the case $\K=\Q$ of Theorem \ref{main}. We prove first a
statement on the existence of small and highly singular elements in maximal
orders.

\begin{theorem} \label{rankone}
Let $\A$ be a $\Q$-subalgebra of $M_n(\R)$
isomorphic to $M_n(\Q)$ and let $\Lambda$  
be a maximal $\Z$-order in $\A$.
Then there exists an element $C\in \Lambda$  which has rank 1 as a matrix, 
and whose Frobenius norm $\|C\|$ is less than $n$.
\end{theorem}

\medskip

\begin{remark}
When we apply the above theorem, the Frobenius norm $\|\cdot \|$  
will be inherited from  $M_n(\R)$, with respect
to an {\em arbitrary}  embedding of $\A$ into $M_n(\R)$. Recall that for a 
matrix $X\in M_n(\R)$ we have  $\| X\| =\sqrt{Tr(X^TX)}$.  
\end{remark}

\medskip

\begin{proof}
The isomorphism $\A\cong M_n(\Q)$ extends to an automorphism
of $M_n(\R)$. Therefore, by the Noether-Skolem Theorem, 
there exists a matrix $P\in M_n(\R)$
such that $\A=PM_n(\Q)P^{-1}$. Let $\Lambda'$  
denote the standard maximal order $M_n(\Z)$ in $M_n(\Q)$. 
The 
theory of maximal orders in central simple algebras over $\Q$ implies that there exists an 
invertible rational matrix $P'\in M_n(\Q)$  such that it gives us 
$P^{-1}\Lambda P$ from $\Lambda'$: 
$P^{-1}\Lambda P=P' \Lambda'{P'}^{-1},$ whence
$\Lambda =PP' \Lambda'{P'}^{-1}P^{-1}. $
Set $Q=PP'/(|\det P\det P'|)^{1/n}$. Clearly $Q\in M_n(\R)$, 
$\det Q$ is $\pm 1$ and 
$$\Lambda = Q \Lambda'Q^{-1}. $$

Let $\rho$ denote the  left ideal of $\Lambda '$ consisting of all 
integer matrices which have 0 everywhere except in the first column.
Clearly $\rho$ is a lattice of determinant 1 in the linear space $S$ of all 
real matrices having nonzeros only in the first column. 
Now the lattice $L=Q\rho$ will be a
sublattice of $S$, with determinant 1 (see Subsection 2.2 from \cite{Ma1}
for basic facts on lattices in real Euclidean spaces).

We can apply Minkowski's theorem on lattice points in convex bodies 
to $L$ in $S$, and to the ball of
radius $\sqrt{n}$ in $S$ centered at the zero matrix 
(we refer here to the Euclidean
distance, that is, the Frobenius norm on $M_n(\R)$).  The volume (calculated in $S$) of
the ball is more than $2^n$, as it contains $2^n$ internally disjoint copies
of the $n$-dimensional unit cube, and more. We infer that there exists an
element $B\in \rho$ such that $QB$ is a nonzero matrix whose length is less
than $\sqrt{n}$.  Clearly $B$ and hence $QB$ is a rank 1 matrix.

Next consider the "transpose" of this argument with $Q^{-1}$ in the place of
$Q$: there exists a nonzero integer matrix $B'$, which is zero everywhere
except  in the first row, such that $B'Q^{-1}$ is nonzero, and has Euclidean 
length less than $\sqrt{n}$.

Now 
$$ C=QBB'Q^{-1} $$ 
meets the requirements of the statement. Indeed, it is in
$\Lambda$ because $BB'\in M_n(\Z)$. It has length less than $n$ because the 
Frobenius norm is submultiplicative:
$$ \| C\| =\|(QB)(B'Q^{-1})\|\leq \|QB\| \cdot \|B'Q^{-1}\|< (\sqrt{n})^2=n.$$
Obviously, $C$ has rank at most 1, as $B$ and $B'$ are  of
rank 1. Finally, from the shape of $B$ and $B'$ we see, that $BB'\not= 0$, hence 
$ \rank \, BB'= \rank \, C= 1$. This finishes the proof.
\end{proof}

\medskip

\medskip

\noindent
{\bf Remark.} Essentially the above reasoning shows the existence of a rank
one $C\in \Lambda$ such that $\|C \|\leq \gamma_n$, where $\gamma_n$ is
Hermite's constant  (see Chapter IX, \cite{Ca1}).
This bound is achieved if we
select  $B$ and $B'$ whose norm is at most most $\sqrt{\gamma_n}$. This 
gives
a better bound for large values of $n$.

\medskip

The following two lemmas point out that elements $X$ form an 
order $\Lambda
\subset M_n(\Q)$ with $\|X\|$ small are necessarily zero divisors. 

\medskip

\begin{lemma} \label{singular} Let $X\in M_n(\C)$ be a matrix such that 
$\det X$ is an 
integer, and $\|X\|<\sqrt{n}$. Then $X$ is a singular matrix. 
\end{lemma}

\begin{proof} The argument is essentially from \cite{AL}. Let $X=QR$ be the 
QR decomposition of $X$, with $Q$ unitary 
and $R$ an upper triangular  matrix whose diagonal entries are $r_1,
r_2, \ldots ,r_n$. We have 
$$ |\det X|^{2/n}=(|r_1|^2|r_2|^2\cdots |r_n|^2)^{1/n}\leq \frac 1n
(|r_1|^2+|r_2|^2+\cdots |r_n|^2)\leq \frac 1n \|R\|^2= \frac 1n \|X\|^2<1.$$
Here we used the fact that $\|X\|=\sqrt{Tr(X^*X)}=\sqrt{Tr(R^*R)}$ because
$Q^*Q=I$. We conclude that $\det X=0$.
\end{proof}

The next statement has a similar flavour. It was pointed out to us 
by our colleague G\'eza K\'os.

\begin{lemma} \label{nilpotent} 
Let $X\in M_n(\Q)$ be a matrix whose characteristic polynomial 
has integral coefficients, and $\|X\|<1$. Then $X$ is a nilpotent matrix.
\end{lemma}

\begin{proof} The eigenvalues of $X$ are algebraic integers, hence the
eigenvalues of $X^t$ are algebraic integers as well, for any positive
integer $t$. We infer that the characteristic polynomial of $X^t$ has
integral coefficients. Also, the norm condition implies that $X^t$ tends to
the zero matrix $O$ as $t\rightarrow \infty $, hence $X^t=O$ for a
sufficiently large $t$.
\end{proof}

\medskip

The following argument is from H. W. Lenstra, see  p. 546 in \cite{Le}.
Informally, it states that the coefficients with respect to a reduced basis 
of a vector $\ve v$  with small length $|\ve v|$ from a lattice 
$\Gamma$ are relatively small.

\begin{lemma} \label{coefficients}
Let $\Gamma$ be a full lattice in $\R^m$. 
Suppose that we have a basis $\ve b_1,\ldots, \ve  b_m$ of $\Gamma $ over 
$\Z$ such that 
\begin{equation}
\label{basiseq}
 |\ve b_1|\cdot |\ve b_2|\cdots |\ve b_m|\leq c_m\cdot \det(\Gamma) 
\end{equation}
holds for a real number $c_m>0$. Suppose that 
$$\ve v =\sum_{i=1}^m\gamma_i \ve b_i \in \Gamma,~~~\gamma_i\in \Z. $$ 
Then we have 
$|\gamma_i|\leq c_m \frac{|\ve v |} {|\ve b_i|}$ for $i=1,\ldots, m$. 
\end{lemma}

\begin{proof}
From Cramer's rule we obtain 
$$ |\gamma_i|=\frac{|\det (\ve b_1,\ve b_2, \ldots, \ve b_{i-1},\ve v, 
\ve b_{i+1}, 
\ldots
,\ve b_m)|}{\det(\Gamma)}\leq \frac{|\ve b_1|\cdots |\ve b_{i-1}|\cdot
|\ve v |\cdot
|\ve b_{i+1}|\ldots |\ve b_m|}{\det (\Gamma)}  =$$.
$$= \frac{|\ve v|}{|\ve b_i|}\cdot \frac{|\ve b_1|\cdots |\ve b_{i-1}|
\cdot |\ve b_i|\cdot
|\ve b_{i+1}|\cdots |\ve b_m|}{\det (\Gamma)}\leq \frac{|\ve v|}{|\ve
b_i|}\cdot c_m \cdot
\frac
{\det (\Gamma)}{\det (\Gamma)}= c_m \cdot \frac{|\ve v|}{|\ve b_i|}.$$
\end{proof}

We remark that the LLL algorithm gives a basis with $c_m=2^{m(m-1)/4}$ in 
formula (\ref{basiseq}), see \cite{LLL}. We shall have a lattice of vectors with nonrational 
coordinates, and thus invoke the approximate version of the LLL algorithm 
developed by Buchmann, see Corollary 4 of \cite{Bu1}. This will provide 
a reduced basis with 
\begin{equation} \label{red}
c_m:= \left(\gamma_m \right)^{\frac m2}\left( \frac 32 \right)^m
2^{\frac{m(m-1)}{2}}.
\end{equation}
Here $\gamma_m$ is Hermite's constant. It is 
known that $\gamma_m\leq m$ for all integers $m\geq 1$, and 
$\frac{\gamma_m}{m}\leq \frac{1}{\pi e}+o(1)$ for $m$ large.  

\medskip

We can describe now the algorithm of Theorem \ref{main} for the case
$\K=\Q$. 
Suppose that, as input, we have an algebra ${\cal A}$ over $\Q$, given to us
by structure constants. Suppose also that ${\cal A}$ is isomorphic to the 
full matrix algebra $M_n(\Q)$. Our objective is to give this isomorphism 
explicitly. More specifically the algorithm outputs an element $C\in \A$
which has rank 1 in $M_n(\Q)$. Then the left action of $\A$ on $\A C$
provides an $\A \rightarrow M_n(\Q)$ isomorphism.  
The major steps of the algorithm are the following. 

\hrulefill

\begin{enumerate}

\item 
Use the Ivanyos-R\'onyai algorithm \cite{IR} to construct a maximal order $\Lambda $
in ${\cal A}$. This is a polynomial time ff-algorithm\footnote{It performs
well if the integers to be factored are not very big. The method has been 
implemented in Magma by de Graaf.}.

\item
Compute an embedding of ${\cal A}$ into $M_n(\R)$. One uses here
the the deterministic polynomial time algorithm obtained via the
derandomization by 
de Graaf and Ivanyos \cite{GI} of the  
Las Vegas algorithm of Eberly 
\cite{Eb3}. 
This way we have a Frobenius norm on  ${\cal A}$. For $X\in {\cal A}$ 
we can set $\|X\|=\sqrt{Tr(X^T X)}$. Also, via this embedding $\Lambda $ can
be viewed as a full lattice in $\R^m$, where $m=n^2$. The length $|\ve v|$ of 
a lattice vector $\ve v$ is just the Frobenius norm of $\ve v$ as  a matrix.

\item 
Compute a rational approximation $A$ of our basis $B$ of $\Lambda$ with
precision $q_0(B,\frac 12, 2^{\frac{m-1}{2}})$ (see Section 2 in \cite{Bu1}
for the definition of the precision parameter $q_0$).  One can use here the
Algorithm of Sch\"onhage\footnote{For a more recent method see 
\cite{Pa1}.}  \cite{Sch1}.

\item 
Compute a reduced basis $\ve b_1,\ldots ,\ve b_m$ of the lattice
$\Lambda\subset \R^m$
by applying the LLL algorithm to $A$. For $c_m$ we have the value 
from (\ref{red}).

\item
If some of the basis elements $\ve b_i$ is a zero divisor in ${\cal A}$,
then there are two cases. If $\rank \, \ve b_i=1$, 
then we are done and stop with the output $C:=\ve b_i$. 
Otherwise, if $1<\rank \, \ve b_i< n$, then  we compute the the right 
identity element $e$ of the 
left ideal $\A \ve b_i$ by solving the straightforward system of 
linear equations, 
set $\A :=e\A e$ and go back to Step 1. 

\item
At this point we know that $|\ve b_i |\geq \sqrt n$ holds for every $i$.
Generate all integral linear combinations $C'=\sum _{i=1}^m\gamma_i\ve b_i$, where the
$\gamma_i$ are integers, $|\gamma_i|\leq 
c_m \frac{n}{|\ve b_i|}\leq c_m \sqrt{n}$
until a $C$ 
is found with $\rank \, C=1$. Output this $C$.

\end{enumerate}

\hrulefill

%-------------------------------------------------------------------
\begin{proof}[Proof of theorem \ref{main} for $\K=\Q$.] As for the correctness of the
algorithm, let $\ve b_1,\ldots ,\ve b_m$ the basis of $\Lambda$ obtained 
at Step 4 with $\| \ve b_1\|\leq \cdots \leq \|\ve b_m\|$. Then by Corollary
4 from \cite{Bu1} we have 
$$\|\ve b_i\|\leq  \frac 32 \cdot
2^{\frac{m-1}{2}}\lambda_i \text{~~for~} i=1,\ldots ,m,$$
where $\lambda_i$ is the $i$-th successive minimum of $\Lambda$. 
From this we infer
$$\| \ve b_1\|\|\ve b_2\| \cdots \|\ve b_m\|\leq 
\left( \frac 32 \right)^m
2^{\frac{m(m-1)}{2}}\lambda_1\cdots \lambda_m\leq 
\left(\gamma_m \right)^{\frac m2}\left( \frac 32 \right)^m
2^{\frac{m(m-1)}{2}}\det (\Lambda ),$$
as claimed. At the last inequality we used Minkowski's inequality on 
successive minima (see
Chapter VIII in \cite{Ca1}).

We remark also that, if at Step 5 we have $\rank \, e= k$, then 
it is easy to see that $e\A e\cong M_k(\Q)$.  Moreover, a rank one element 
of $e\A e$ will have rank one in $\A$ as well.   At Step 6 the $\ve b_i $
are nonsingular matrices, hence   $\|\ve b_i \|\geq \sqrt n$ holds by 
Lemma \ref{singular}. Finally, Theorem \ref{rankone} and 
Lemma \ref{coefficients} (this is applied for $\ve v :=C$ and 
$|\ve v|\leq n$) show that an element 
$C$ with rank one exists among the linear combinations enumerated.

Considering the timing of the algorithm, Step 1 runs in polynomial time as
an ff-algorithm. Steps 2, 4 and 5 can be done in deterministic polynomial 
time. At Step 3 the precision paramater $q_0$ is polynomial in the input
size, hence Sch\"onhage's approximation algorithm 
(see also Section 3 of \cite{KLL}) runs in polynomial time.

The number of 
jumps back to Step 1 
is also bounded, hence each Step is carried out in a bounded number of times. 
Finally, the number of elements $C'$ enumerated at Step 6 is at most 
$(2 c_m \sqrt{n}+1)^m$, this is also bounded by our assumption.
\end{proof}

\medskip

\noindent
{\bf Remarks.} 1. In Step 4 of the preceding algorithm one may also consider 
the idempotent $f=I-e$, where $I$ is the identity element of $\A$. If $\rank
f=1$, then we can stop with $C:=f$. Otherwise, if $\rank f < \rank e$, then 
we may work with $f\A f$ instead of $e\A e$. \\
2. We could avoid jumps back to Step 1 if we had a good lower bound on the 
quantities $\|\ve b_i\|$. Unfortunately, we do not have such a bound in 
general. The difficulty here 
may come from the fact, that the closure of the similarity-orbit of
nilpotent matrices contains the zero matrix. This is illustrated by 
the matrices
$$ X=\left( \begin{array}{cc}t & 0 \\
                             0 & \frac 1t
             \end{array} \right),  ~~~ E=
\left( \begin{array}{cc}0 & 1 \\
                             0 & 0
             \end{array} \right). $$
We have $ XEX^{-1}= t^2E$, hence $\|XEX^{-1}\|$
gets arbitrarily small as $t\rightarrow 0$. \\
3. We could have used Lemma \ref{nilpotent} instead of Lemma \ref{singular}. 
In this case we test in Step 4 if there is a nilpotent element among the  
$\ve b_i$.  Also, then in Step 5 we have to enumerate integral linear 
combinations $\sum _{i=1}^m\gamma_i\ve b_i$
with   $|\gamma_i|\leq c_m \cdot n$.

\section{The general case}

Let $\K$ be a number field of degree $d$ over $\Q$, the maximal order 
of $\K$ is denoted by $R$ and the positive discriminant of $R$ is $\Delta$.
Let $\A$ be a central simple algebra over $\K$ such that 
$\Aalg\cong M_n(\K)$, and let $\Lambda$ be a maximal order in 
$\A$.

It is known (see Reiner \cite{Re}, Corollary 27.6) that
there is an isomorphism $\psi:\Aalg\rightarrow M_n(\K )$
such that the image of $\Lambda$ is
$$
\Lambda':=\psi(\Lambda)=\begin{pmatrix}
R & \cdots R & J^{-1} \\
\vdots & \ddots & \vdots  \\
R & \cdots R & J^{-1} \\
J & \cdots J & R 
\end{pmatrix},
$$
where $J$ is a fractional ideal of $R$ in $\K$. (The notation with a 
matrix having sets as entries refers to all matrices $(x_{ij})_{i,j=1}^n$ 
whose elements belong to the designated sets, for example, $x_{11}\in R$, 
$x_{n1}\in J$, etc.) 
Let $\sigma_1,\ldots,\sigma_r$ be the embeddings
of $\K$ into $\R$ and
$\sigma_{r+1},\overline{\sigma_{r+1}},
\ldots,\sigma_{r+s},\overline{\sigma_{r+s}}$ be 
the non-real embeddings of $\K$ into $\C$; here we have $d=r+2s$.

For each $1\leq i\leq r+s$ let us 
consider an 
embedding $\phi_i$ of $\Aalg$ into $M_n(\C)$, which extends
 $\sigma_i$
(for $i\leq r$ we require $\phi_i(\Aalg)\leq M_n(\R)$). We remark  
that such embeddings can actually be computed efficiently by the methods 
of \cite{Eb3} and \cite{GI}. For $x\in \A$ the matrices  $\phi_i(x)$ are 
in $M_n(\C)$, hence we may speak about the absolute value of their entries. 
Set 
$$
b=\left(\left(\frac{2}{\pi}\right)^{2sn}\Delta^{n}\right)^{\frac{1}{nd}}
=\left(\frac{2}{\pi}\right)^{\frac{2s}{d}}\Delta^{\frac{1}{d}}.$$

\begin{theorem} \label{rankone1}
There exists a rank one element $x\in \Lambda $ such that 
the entries of the matrices $\phi_i(x)$ for $i=1, \ldots ,s+r$ all have 
absolute value at most $b$. 
\end{theorem}

\begin{proof}
Let $\psi_i:\Aalg\rightarrow M_n(\C)$ be the composition
of $\psi$ with the natural extension of $\sigma_i$ to
$M_n(\C)$. 
These maps are shown at the diagram below. The vertical map is the 
extension of $\sigma_i$ from $\K$ to $M_n(\K)$. The triangle is commutative.
$$ 
\begin{array}[c]{ccccc}
M_n(\C)&\xleftarrow{\phi_i} & \A & \xrightarrow{\psi} & M_n(\K) \\
&&& \stackrel{\psi_i}\searrow &\downarrow ^{\sigma_i} \\
&&&& M_n(\C)    
\end{array}
$$
Then the $\C$-linear extensions of
the composite maps $\phi_i\psi_i^{-1}$ from
$\psi_i(\Aalg)$ to $M_n(\C)$
are $\C$-algebra automorphisms of $M_n(\C)$
(whose restrictions, for $i=1,\ldots,r$,
to the real matrices are automorphisms of $M_n(\R)$).
As these automorphisms must be inner, 
there exist matrices $A_1,\ldots,A_r\in GL_n(\R)$ with determinant $\pm 1$ 
and $A_{r+1}, \ldots, A_{r+s}\in SL_n(\C)$ 
such that for $i=1,\ldots,r+s$ we have
$$\phi_i(\Lambda)=A_i^{-1}\Lambda ' A_i=
A_i^{-1}
\begin{pmatrix}
\sigma_i(R) & \cdots \sigma_i(R) & \sigma_i(J^{-1}) \\
\vdots & \ddots & \vdots  \\
\sigma_i(R) & \cdots \sigma_i(R) & \sigma_i(J^{-1}) \\
\sigma_i(J) & \cdots \sigma_i(J) & \sigma_i(R)
\end{pmatrix}A_i.
$$
Put $A'_i:=(A_i^{-1})^T$.
We show that there exist nonzero vectors
$\ve u\in (R,\ldots,R,J)\subset \K^n$ and $\ve v\in
(R,\ldots,R,J^{-1})\subset \K^n$
such that for every index $i=1,\ldots,r$,
all the coordinates of $\sigma_i(\ve u)A_i'$ and $\sigma_i(\ve v)A_i$ 
are of "small" absolute values. Then all the entries of the matrix 
$\phi_i\psi^{-1}(\ve u^T\ve v)$ will be small, demonstrating that there exist
a rank one element of $\Lambda$, namely  
$\psi^{-1}(\ve u^T\ve v)$, 
which is small in all the
embeddings $\phi_i$.

To this end,  we consider the set $\cal M$
of row vectors of length $nd$ of the form
\begin{equation} \label{embed}
(\sigma_1(\ve u),\ldots,\sigma_r(\ve u),
\sigma_{r+1}(\ve u),\overline{\sigma_{r+1}}(\ve u),
\ldots,
\sigma_{r+s}(\ve u),\overline{\sigma_{r+s}}(\ve u)),
\end{equation}
where
$\ve u\in (R,\ldots,R,J)$.
$\cal M$ is a lattice in the linear space $\C^{dn}$ whose rank is $nd$
because of the linear independence of field automorphisms, see Theorem I.3
in \cite{Jac}. The  
determinant of lattice ${\cal M}$ is 
$$\Delta^{n/2}N(J),$$
where  $N(J)$ is the norm of the fractional ideal $J$ (see Proposition 13.4,
Chapter I. in 
\cite{Jan}).
Next we consider the set $\cal M'$ of vectors
of the form
$$(\sigma_1(\ve u)A_1',\ldots,\sigma_r(\ve u)A_r',
\sigma_{r+1}(\ve u)A_{r+1}',\overline{\sigma_{r+1}}(\ve u)\overline{A_{r+1}'},
\ldots,
\sigma_{r+s}(\ve u)A_{r+1}',
\overline{\sigma_{r+s}}(\ve u)\overline{A_{r+s}'}).$$
This set is obtained by multiplying vectors from 
$\cal M$ by the block diagonal matrix 
$$diag\left(A_1',\ldots,A_r',A_{r+1}',\overline{A_{r+1}'},
\ldots,A_{r+1}',\overline{A_{r+s}'}\right).$$
Here each block has determinant $\pm 1$,
therefore the determinant of $\cal M'$ remains $\Delta^{n/2}N(J)$.

Finally we apply the block diagonal matrix
$$diag\left(I,\ldots,I,
\begin{pmatrix}
\frac{1}{2}I & \frac{-\iota}{2}I \\
\frac{1}{2}I & \frac{\iota}{2}I
\end{pmatrix},
\ldots,
\begin{pmatrix}
\frac{1}{2}I & \frac{-\iota}{2}I \\
\frac{1}{2}I & \frac{\iota}{2}I
\end{pmatrix}
\right),
$$ 
where $I$ stands for the $n$ by $n$ identity matrix, and we have 
$r$ blocks of $I$.
The determinant of this matrix is $(\iota /2)^{ns}$.
From ${\cal M'}$ we obtain the lattice $\cal L$ of rank $nd$ in 
$\R^{nd}\subset \C^{nd}$ 
consisting of the vectors
\begin{equation} \label{embed2}
(\sigma_1(\ve u)A_1',\ldots,\sigma_r(\ve u)A_r',
\Re(\sigma_{r+1}(\ve u)A_{r+1}'),
\Im(\sigma_{r+1}(\ve u)A_{r+1}'),
\ldots,
\Re(\sigma_{r+s}(\ve u)A_{r+1}'),
\Im(\sigma_{r+s}(\ve u)A_{r+1}'),
\end{equation}
where $\ve u$ runs over $(R,\ldots,R,J)\subset \K^n$.
The determinant of $\cal L$ is $2^{-sn}\Delta^{n/2}N(J)$.
We apply now Minkowski's theorem on convex bodies
to the lattice $\cal L$
and to the product of $rn$ one-dimensional balls
and $sn$ two-dimensional balls of radius
$$r(J)=\left(\left(\frac{2}{\pi}\right)^{sn}N(J)
\Delta^{n/2}\right)^{\frac{1}{nd}}.$$
This is a closed convex centrally symmetric (with respect to the origin) 
body of volume
$$ \left(2r(J)\right)^{rn}\left(\pi r(J)^2\right)^{sn}.$$
This volume is $2^{nd}\det {\cal L}$. The theorem 
tells us that there exists a nonzero 
$\ve u\in(R,\ldots,R,J)$ such that 
for every $1\leq i\leq r+s$,
all the coordinates
of $\sigma_i(\ve u)A_i'$ have absolute value
at most $r(J)$.

Similarly, there exists a nonzero vector
$\ve v\in(R,\ldots,R,J^{-1})$ such that
for every $1\leq i\leq r+s$,
all the coordinates of $\sigma_i(\ve v)A_i$ have absolute value
at most $r(J^{-1})$ where
$$r(J^{-1})=\left(\left(\frac{2}{\pi}\right)^{sn}N(J)^{-1}
\Delta^{n/2}\right)^{\frac{1}{nd}}.$$

Then $x=\psi^{-1}(\ve u^T\ve v)$ is a rank one element of $\Lambda$ such that
for every $i$, all the entries of the
matrix $\phi_i(x)$ have absolute value
at most 
$$
r(J)r(J^{-1})=\left(\left(\frac{2}{\pi}\right)^{2sn}
\Delta^{n}\right)^{\frac{1}{nd}}
û=\left(\frac{2}{\pi}\right)^{\frac{2s}{d}}\Delta^{\frac{1}{d}}=b.$$
\end{proof}

\medskip

We point out two interesting consequences: \\
1. If $\K =\Q$, $R= \Z$, then $\Delta=1$, $s=0$, hence $b=1$. 
We have an element $x$ of our maximal order $\Lambda$  which has rank 1 
as a matrix from $ M_n(\Q)$, and with respect to our selected embedding of 
$\Aalg$ into $M_n(\R)$ has elements of absolute value at most 1. This 
is essentially Theorem \ref{rankone}. \\
2. If $D$ is a positive squarefree integer, $\K =\Q(\sqrt{D})$, 
then $\Delta=D$, if $D$ is congruent to 1
modulo 4, and $\Delta=4D$, if $D$ is congruent to 3 modulo 4.
Then  $s=0$, $d=2$, hence $b\leq 2\sqrt{D}$.

To our algorithm we shall need a more general variant of Lemma \ref{singular}. 

\begin{lemma} \label{singular1} Let $y\in \Lambda $ be an element such that 
$\|\phi_i(y)\|< \sqrt{n}$ holds for $i=1,\ldots ,r+s$. Then $y$ is a zero
divisor in $\A$. 
\end{lemma}

\begin{proof} As in Lemma \ref{singular} we obtain that 
\begin{equation} \label{determinant}
|\det \phi_i(y) |< 1 \text{~for~} i=1,\ldots ,r+s. 
\end{equation}

Note that $\det \phi_i(y)=\sigma _i(n(y))$, where $n(y)$ is the reduced 
norm of $y$ (see Section 9 in \cite{Re}). 
Inequality (\ref{determinant}) implies that 
$$  | \sigma_1(n(y)) \cdots  \sigma_r(n(y)) \sigma_{r+1}(n(y))  
\overline{\sigma_{r+1}(n(y))} \cdots 
\sigma_{r+s}(n(y))  \overline{\sigma_{r+s}(n(y))} | <1.$$
Moreover, by Theorem 10.1 from \cite{Re} $n(y)\in R$, 
therefore the number on the left is a rational integer, giving that  
$\det \phi_i(y)=0$
for at least one (and hence for all) $i$. This implies that $y$ is a zero
divisor in $\A$. 
\end{proof}

To be able to use lattice basis reduction techniques, we use a
transformation which turns a maximal order in $\A$ into a full lattice in a 
suitable real linear space. To this end for $y\in \Lambda $ we form the 
vectors 
$$\Phi(y):= (\phi_1(y),\ldots,\phi_r(y),
\Re(\phi_{r+1}(y)),
\Im(\phi_{r+1}(y)),
\ldots,
\Re(\phi_{r+s}(y)),
\Im(\phi_{r+s}(y))).
$$
As with (\ref{embed}) and (\ref{embed2}), we infer that 
$\Gamma:=\Phi(\Lambda)$ is a 
full lattice in  the real linear space in $\R^m$, with $m=n^2d$.

We give
now the algorithm of Theorem \ref{main} for the general case:
as input, we have an algebra ${\cal A}$ over $\K$, given to us
by structure constants. Suppose further, that ${\cal A}\cong M_n(\K)$. 
Our algorithm outputs an element $x\in \A$
which has rank 1 in $M_n(\K)$. 

\hrulefill

\begin{enumerate}

\item 
Use the Ivanyos-R\'onyai algorithm \cite{IR} to construct a maximal order
$\Lambda $
in ${\cal A}$. 

\item
Compute the embeddings $\phi_i$ of ${\cal A}$ into $M_n(\C)$ for
$i=1,\ldots, r+s$ (they are embeddings into $M_n(\R)$ for $i\leq r$)
by the deterministic  variant  \cite{GI} of Eberly's algorithm \cite{Eb3}.

\item
Form a basis of the full rank lattice 
$\Gamma\subset \R^m$ with $m=n^2d$. Note that for the 
Euclidean length in $\Gamma $ we have 
$$|\Phi(y)|^2=\sum_{i=1}^{r+s}\|\phi_i(y)\|^2. $$

\item 
Compute a reduced basis $\ve b_1,\ldots ,\ve b_m$ of the lattice
$\Gamma\subset \R^m$
by using Buchmann's approximate version the LLL algorithm to achieve 
the value in (\ref{red})
for the reducedness factor $c_m$.

\item
If an element $y=\Phi^{-1}(\ve b_i)$ is a zero divisor in ${\cal A}$,
then there are two cases. If $\rank \, y=1$, 
then we are done and stop with the output $x:=y$. 
Otherwise, if $1<\rank \, y< n$, then  we compute the the right
identity element $e$ of the 
left ideal $\A y$, 
set $\A :=e\A e$ and go back to Step 1.

\item
At this point we know that $|\ve b_i |\geq \sqrt n$ holds for every $i$.
Generate all linear combinations $\ve w=\sum _{i=1}^m\gamma_i\ve b_i$, where
the
$\gamma_i$ are rational integers with  
$$ |\gamma_i|\leq 
c_m \frac{bn\sqrt{r+s}}{|\ve b_i|}\leq c_m b \sqrt{n(r+s)}= 
c_m \left(\frac{2}{\pi}\right)^{\frac{2s}{d}}\Delta^{\frac{1}{d}}
\sqrt{n(r+s)}$$
until a $\ve w$ is found such that $\rank\, x=1$ holds for the $x\in \Lambda$
with $\Phi(x)=\ve w$. Output this $x$.

\end{enumerate}

\hrulefill

\begin{proof}[Proof of Theorem \ref{main}.]
The proof is essentially the same as in the simpler case $\K=\Q$. 
At Step 6  $\Phi^{-1}(\ve b_i) $
is necessarily a nonsingular element of $\Lambda$ for $i=1,\ldots, r+s$. 
By Lemma \ref{singular1} there must be a $j$ such that 
 $\|\phi_j(\Phi^{-1}( \ve b_i)) \|\geq \sqrt n$, giving that 
$|\ve b_i |\geq \sqrt n$.
Theorem \ref{rankone1} and    
Lemma \ref{coefficients},  the latter is applied with 
$|\ve v|\leq bn\sqrt{r+s}$, show that an element $\ve w$ 
with $\rank \, \Phi^{-1}(\ve w)=1$ 
exists among the linear combinations enumerated.

Here also each Step is carried out 
in a bounded number of times. 
The number of elements $\ve w$ enumerated at Step 6 is at most
$(2 c_m b\sqrt{n(r+s)}+1)^m$. This is also bounded by our assumptions.
\end{proof}

\section{Two consequences}

From the elementary theory of the Brauer group (see for example Section 12.5
from \cite{Pi}) we know that for two central simple algebras 
$\A$ and $\B$ of the same dimension $n^2$ over a field $\K$ we have 
$\A\cong \B$ if and only if 
\begin{equation} 
\label{iso}
\A\otimes_{\K} \B^{op}\cong M_{n^2}(\K).
\end{equation}
We outline next that, over an infinite $\K$,  how one can efficiently recover 
from an isomorphism (\ref{iso}) an isomorphism 
$\sigma: \A \rightarrow \B$. 

Having isomorphism (\ref{iso}) explicitly implies that we have 
in our hands a left $\A \otimes_{\K} \B^{op}$-
module $V$ of dimension $n^2$ over $\K$. 
Then $V$, as a left $\A$-module, is isomorphic
to the regular left $\A$-module because they have the same dimension over
$\K$.  There exists an element $v\in V$ such that the map 
$\phi_v:a\mapsto  av$ is 
a left $\A$-module isomorphism from $\A$ to $V$. The elements
$v$ of $V$ which do not generate $V$ as a left $\A$-module are
zeros of a certain polynomial on $V$ of degree $n^2$ (the determinant
of the linear map $a\mapsto av$). Similarly, the elements
$v$ of $V$ for which the map $\psi_v:b\mapsto vb$
is not a right $\B^{op}$-module isomorphism between $\B^{op}$ and $V$
are zeros of a polynomial on $V$ of degree $n^2$. Therefore,
by the Schwartz-Zippel Lemma there exists an 
element $v\in V$ for which the maps $\phi_v$ and $\psi_v$
are simultaneously left and right isomorphisms, respectively.
The methods of \cite{BR} or \cite{CIK97} for finding large 
cyclic submodules can be used to obtain first a left
$\A$-module generator $V$ and then essentially the same
method can be applied to gradually transform $v$ to a generator
of $V$ as a right $\B^{op}$-module while preserving  the property that
$v$ is a left $\A$-module generator for $V$. For example, the the method 
of Lemma~8 from~\cite{CIK97} can be used here. We recall the statement of 
the lemma for the reader's convenience.

\begin{lemma}
Let $V$ be an $r$-dimensional  module
over the semisimple $\K$-algebra $\A$ and
$v_1,\ldots,v_r$ be a $\K$-basis of $V$. 
Assume that $v\in V$ is an element of non-maximal rank.
Let $\Omega$ be a subset of $\K^*$ consisting of 
at least $\mbox{\rm rk}\, v+1$ elements. 
Then there exists a scalar $\omega\in\Omega$
and a basis element $u\in\{v_1,\ldots,v_r\}$
such that $\mbox{\rm rk}(v+\omega u)>\mbox{\rm rk}\, v$.
(Here the rank $\mbox{\rm rk}\, v$ of $v$ is defined
as the dimension of the $\A$-submodule of $V$
generated by $v$.)
\end{lemma}

We claim that if $v\in V$ is an element such that $\phi_v$ and $\psi_v$
are simultaneously isomorphisms of the respective modules, 
then $\sigma=\psi_v^{-1}\phi_v$ is an
algebra isomorphism between $\A$ and $\B$. It is obvious that
$\sigma$ is a $\K$-linear isomorphism between $\A$ and $\B$. Note that
for $a\in \A$, $\sigma a$ is the unique element $b\in \B$ with
the property $av=vb$. Therefore $\sigma(a_1a_2)$ is the unique
element of $\B$ with $a_1a_2v=vb$. But $a_1a_2v=a_1v(\sigma a_2)=
v(\sigma a_1)(\sigma a_2)$, whence $\sigma(a_1a_2)=(\sigma a_1)(\sigma
a_2)$.

Combining this argument with the algorithm of 
Theorem \ref{main}
for constructing a suitable module $V$, we obtain the following: 

\begin{corollary} \label{application}
Let $\K$ be an algebraic number field of degree $d$ and discriminant $\Delta$
over $\Q$. Let 
$\A, \B$ be central simple algebras over $\K$ of the same dimension $n^2$ 
given by structure constants.
Suppose that $d$, $n$ and $|\Delta|$ are bounded. If
$\A$ and $\B$ are isomorphic, then
an isomorphism
$\A \rightarrow \B$ can be constructed by a polynomial time ff-algorithm.
\end{corollary}

\medskip

The next statement is quite modest. It formulates a very plausible claim,
but, to the best of our knowledge, it was not proven before.

\begin{corollary} \label{smallzerodiv}
Let $\K$ be an algebraic number field and 
$\A$ be an associative algebra over $\K$ given by structure constants such
that $\A\cong M_n(\K)$ holds for some integer $n>1$. Then there exists 
a zero divisor $x\in \A$ which admits polynomially bounded coordinates 
with respect to the input basis of $\A$. Moreover, such a zero divisor 
$x$ can be obtained by a polynomial space bounded computation. 
\end{corollary}

\begin{proof} A slight modification of the algorithm of Therorem \ref{main}
will provide a reasonably small zero divisor: at Step 5 we stop if $y$ is a zero
divisor. Note that $y$ has polynomial size as Steps 1-5 constitute a polynomial
time ff-alghorithm. If no zero divisor is found at Step 5, then we proceed 
directly to Step 6. The integral linear combinations considered there have size 
polynomial in the input length, and their enumeration can be carried out
using polynomial space only.
\end{proof}

\noindent
{\bf Remark.} A more direct, but perhaps algorithmically less efficient 
proof of Corollary \ref{smallzerodiv} is possible. Let 
$\ve c_1, \ldots , \ve c_{n^2}$ be the basis of $\Lambda $ given by the Ivanyos
R\'onyai algorithm. Express the element $x$ of Theorem \ref{rankone1} in
this basis: 
$$ x=\alpha _1 \ve c _1 +\alpha_2 \ve c _2+\cdots + \alpha _{n^2}\ve c _{n^2},
$$
with $\alpha_i\in \Z$. Using that $\|x\|\leq bn$, and that the vectors 
$\ve c _i$ have polynomial size, Cramer's rule implies a 
polynomial bound on the size of the coefficients $\alpha _i$.  

\medskip

By the well known connection between split cyclic algebras and
relative norm equations 
(see Theorem 30.4 in Reiner \cite{Re}), our results imply 
that for
a number field $\K$ and a cyclic extension $\mathbb L$ of $\K$ 
if a norm equation $N_{\mathbb L /\K }(x)=a$ is solvable, then there
is a solution whose standard representation has polynomial size
(in terms of the size of the standard 
representation of $a$ and a basis of $\mathbb L$). Furthermore,
for fixed $\K$ and fixed degree $|\mathbb L : \K|$, a solution
can be found by a polynomial time ff-algorithm.

We have given here a polynomial time ff-algorithm for the explicit 
isomorphism problem for central simple algebras $\A$ of fixed dimension 
over a 
fixed number field $\K$. Potential directions to extend this result 
may be allowing the dimension of the algebra over $\K$ to grow or allowing 
$\K$ to vary (even if its degree over $\Q$ remains fixed), or both. 
Existence of ff-algorithms for finding an explict isomorphism of a 
non-split central simple algebra
with the algebra of matrices over a skewfield is also left open
(even in the case of fixed base field, or fixed dimension). 
It would be interesting also to develop practical variants and programs for 
the algorithms presented here.

\end{document}